\documentclass[letterpaper, 10 pt, conference]{ieeeconf}  % Comment this line out if you need a4paper

\IEEEoverridecommandlockouts                              % This command is only needed if 
\usepackage[T1]{fontenc}
\usepackage{graphicx} % for pdf, bitmapped graphics files
\usepackage{amsmath} % assumes amsmath package installed
\usepackage{amssymb}  % assumes amsmath package installed
\usepackage{mathtools}
\usepackage[hyphens]{url}

\usepackage[ruled,vlined,linesnumbered]{algorithm2e}
\usepackage{algorithmic,algorithm2e,float}

\SetKw{Continue}{continue}

\usepackage{amsthm}
\theoremstyle{definition}
\newtheorem{problem}{Problem}

\newtheorem{definition}{Definition}
\newtheorem{proposition}{Proposition}

\newtheorem{exmp}{Example}
\newtheorem*{rem}{Remark}

\newcommand{\ggr}{$\textsc{Grid-Greedy}$}
\newcommand{\igr}{$\textsc{Centroid-Greedy}$}

\title{\LARGE \bf
An Improved Greedy Algorithm for Subset Selection \\ in Linear Estimation}

\author{Shamak Dutta, Nils Wilde, and Stephen L. Smith% <-this % stops a space
\thanks{The authors are with the Department of Electrical and Computer Engineering,
        University of Waterloo, Canada 
        \{{\tt\small stephen.smith, nwilde, shamak.dutta\}@uwaterloo.ca}.  N.\ Wilde is also with the Cognitive Robotics Department, Delft University of Technology, Netherlands.}%
}
\begin{document}
\maketitle
\thispagestyle{empty}
\pagestyle{empty}

\begin{abstract}
In this paper, we consider a subset selection problem in a spatial field where we seek to find a set of $k$ locations whose observations provide the best estimate of the field value at a finite set of prediction locations. The measurements can be taken at any location in the continuous field, and the covariance between the field values at different points is given by the widely used squared exponential covariance function.
 One approach for observation selection is to perform a grid discretization of the space and obtain an approximate solution using the greedy algorithm. The solution quality improves with a finer grid resolution but at the cost of increased computation. We propose a method to reduce the computational complexity, or conversely to increase solution quality, of the greedy algorithm by considering a search space consisting only of prediction locations and centroids of cliques formed by the prediction locations. We demonstrate the effectiveness of our proposed approach in simulation, both in terms of solution quality and runtime.
\end{abstract}

\section{Introduction}
An important problem in engineering applications is deciding the subset of measurements that are the most useful in the estimation of an unknown quantity of interest. For example, in agriculture, it is important to estimate the nutrient quality of a field using soil samples. This helps guide fertilizer usage to replenish lost nutrients, which subsequently maximizes crop yield. However, it is impractical to sample the soil at each location in large agricultural fields. The goal is to determine where to sample the soil, such that the nutrient quality at a large set of prediction locations can be estimated accurately. An example of the soil pH variability in a field with a set of prediction locations is shown in Figure~\ref{fig:intro}. This type of subset selection problem shows up in domains such as sensor placement/active sampling in spatial statistics \cite{krause2008near,ramakrishnan2005gaussian,le2009trajectory,marchant2012bayesian,yang2017real}, feature selection in machine learning \cite{miller2002subset,guyon2003introduction}, informative path planning in robotics \cite{binney2012branch,binney2010informative,binney2013optimizing,suryan2020learning}, among others. The challenge is similar: choose the subset of attributes that best estimates the quantity of interest.
\begin{figure}
    \centering
    \includegraphics[width=0.45\textwidth]{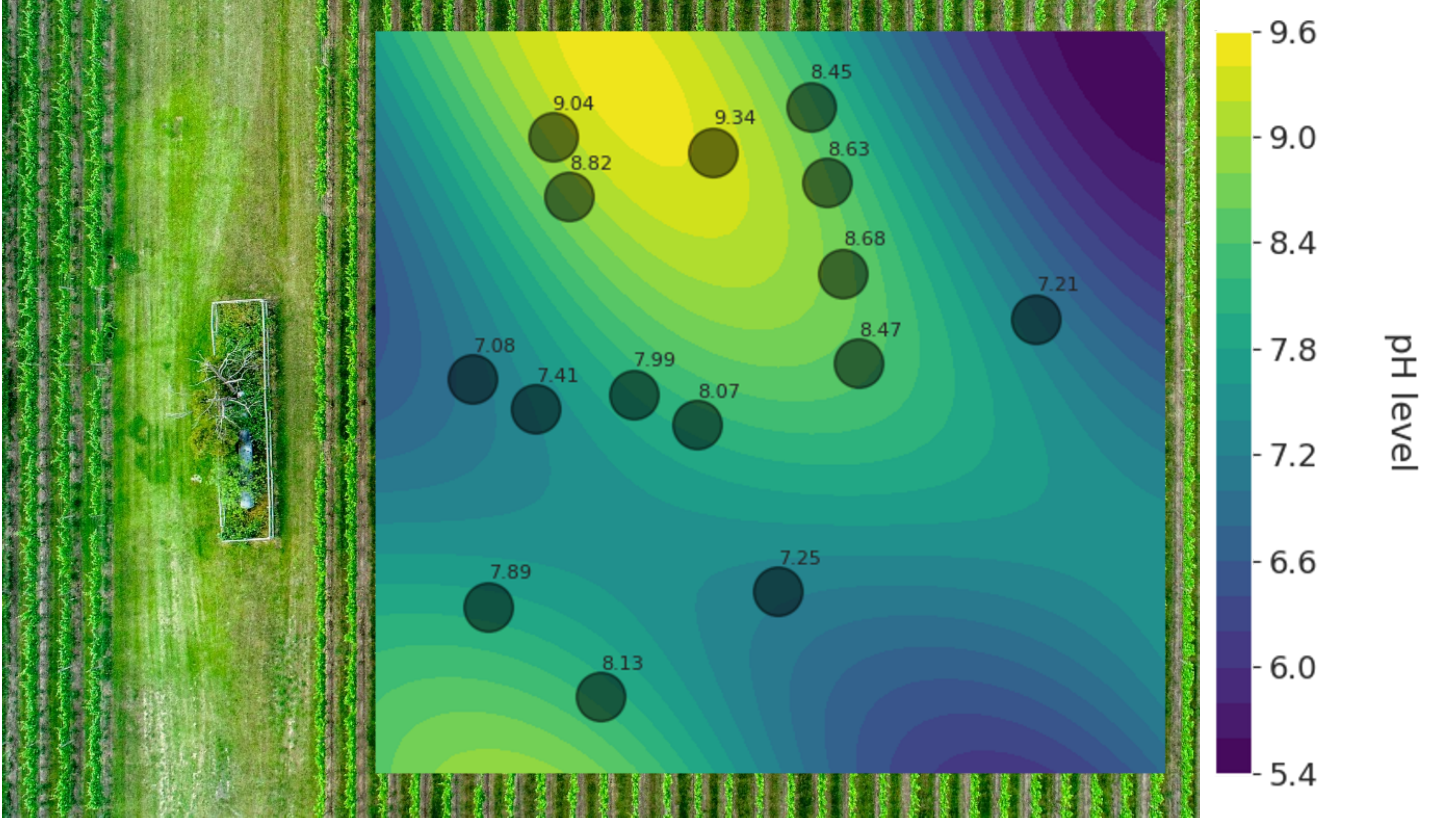}
    \caption{An example of the pH variability in an agricultural field. The circles are the prediction locations where accurate estimates are desired. Agricultural fields can be large and one can only take a fixed number of soil samples to best estimate the pH variability at the prediction locations.}
    \label{fig:intro}
\end{figure}

The Bayesian approach is to model the quantities as random variables. The estimates of prediction variables are obtained by linear estimators and the quality of the chosen subset is measured by the resulting mean squared estimation error. The benefits of this approach are twofold. First, prior statistical knowledge of the quantities can be incorporated into the estimation procedure. Second, the mean squared error resulting from a linear estimator is independent of the observations. Thus, deciding the subset that minimizes the mean squared error can be done \textit{a priori}. With a finite observation set, a popular approach is to use the greedy algorithm \cite{miller2002subset,hastie01statisticallearning}. Each step, the variable maximizing the marginal gain is selected. Continuous observation sets, such as agricultural fields, can be made finite by a grid discretization, which can be used by the greedy algorithm, which we refer to as \ggr. The solution quality improves with a finer grid but at an increased computational cost. The objective of this paper is to remove the dependence on the grid discretization while obtaining good solution quality. Our proposed method, \igr, restricts the search to the set of prediction locations and the centroids of the cliques formed by the prediction locations. This is motivated by our analysis of the problem in one dimension where we identify a critical distance between points that characterize the optimal measurement location. In our experiments, we show \igr\ achieves better solutions when given the same computational resources as \ggr\ and finds solutions of similar quality more efficiently.

\emph{Related Work:}
Similar to our work, \cite{suryan2020learning} study 
the problem of minimizing the number of measurements taken while keeping the estimation error under a threshold in a spatial field modeled as a Gaussian Process with a squared exponential covariance function. In contrast, we minimize the estimation error over a finite set of prediction variables using $k$ observation variables. In addition, the variables are not restricted to be Gaussian in our problem setup.

A seminal paper in the area of sensor placement in Gaussian Processes is \cite{krause2008near} where the authors consider maximizing the mutual information between the sensed and unsensed locations. Using the submodularity of mutual information, the authors use a greedy algorithm to obtain a constant factor approximation guarantee. However, to the best of our knowledge, there is no direct relation between the mutual information and the resulting estimation error. Our work considers the estimation error directly. In \cite{cortes2009distributed}, the author studies the problem of estimation using kriged Kalman filtering in a spatio-temporal field. In our work, we consider spatial variation only and share the kriging aspect with the work in \cite{cortes2009distributed}. 

Another related problem is the subset selection problem in linear regression, where one has to select a subset of $k$ random variables to yield the best prediction of another random variable of interest. The authors in \cite{das2008algorithms} study a special case where the random variables can be embedded onto the real line with covariances that decay with the distance. They provide an algorithm to compute the optimal solution using dynamic programming. However, their work assumes a finite set of measurement variables, which must be embedded onto the real line. There is no known extension of this idea to an infinite set of observation variables embedded in higher dimensional spaces. In this paper, we consider random variables indexed on an infinite subset of a $d$-dimensional Euclidean space with the widely used squared exponential covariance function \cite{rasmussen2003gaussian}. In \cite{das2018approximate}, the authors introduce the concept of approximate submodularity to provide guarantees for the greedy algorithm. The setup is the same as in \cite{das2008algorithms} with a finite set of observation variables. Thus, the performance guarantees are not directly applicable.

Finally, one can also solve the problem using global optimization techniques. In the context of feature subset selection, \cite{narendra1977branch,somol2004fast} study branch and bounds methods to find the optimal subset of features. Building on this idea, the work in \cite{binney2012branch} studies the problem for informative path planning where the goal is to compute a cost-constrained path in a Gaussian Process which minimizes the prediction error. While the algorithm does provide a speedup compared to the brute-force method on small graphs, the worst-case runtime complexity is still exponential in the size of the input. 

\emph{Contributions:} 
The contributions of this work are twofold. First, we formulate a problem of budget constrained observation selection from an infinite set to best estimate a finite set of prediction variables. Second, we propose \igr, a greedy algorithm that uses a ground set consisting of the prediction locations and the centroids of cliques formed by the prediction locations. This reduces the dependence of \ggr\ on the grid discretization of the continuous field. In simulations, we demonstrate the improved solution quality and run time of \igr\ in comparison to \ggr.

\section{Preliminaries} \label{sec:prelim}
Many combinatorial problems involve maximizing a submodular set function whose definition follows.
\begin{definition}[Submodular Set Function] \label{def:submodular}
Given a finite set $\mathcal{V}$, the set function $f: 2^\mathcal{V} \rightarrow \mathbb{R}$ is submodular if for all sets $A \subseteq B \subseteq \mathcal{V}$ and $x \in \mathcal{V} \setminus B$, the diminishing returns property is satisfied:
\begin{equation}
    f(A \cup \{x\}) - f(A) \geq f(B \cup \{x\}) - f(B).
\end{equation}
The greedy algorithm for maximizing set functions subject to a cardinality constraint is described as follows.
\begin{definition}[Greedy Algorithm] \label{def:greedyalg}
 The greedy algorithm begins with the empty set $S_0 = \emptyset$ and repeatedly adds the element $x \in V$ that maximizes the marginal gain until the cardinality constraint is met. That is, for $i \geq 1$,
\begin{equation} \label{background:greedy}
    S_{i+1} = S_{i} \cup \{\underset{x \in V}{\text{arg max}}\ f(S_i \cup \{x\}) - f(S_i)\}.
\end{equation}
\end{definition}
The cardinality constrained maximization of a certain class of submodular functions can be efficiently approximated by a greedy algorithm whose solution is within a multiplicative factor of $1-1/e \approx 0.63$ of the optimal \cite{nemhauser1978analysis}.

Let $X_1, \ldots, X_n, Y$ be square integrable, zero mean random variables, and $\boldsymbol{b} := \left(\text{Cov}(X_1, Y), \ldots, \text{Cov}(X_n, Y)\right)^T$, $\boldsymbol{X} := \left( X_1, \ldots, X_n\right)^T$, and let $C$ be a $n \times n$ matrix whose $(i,j)^{\text{th}}$ element is $\text{Cov}(X_i, X_j)$.
\end{definition}
\begin{definition}[Linear Least Squares Estimator] \label{def:llse}
Given $X_1, \ldots, X_n$, the optimal linear estimator of $Y$ is:
\begin{equation}
\hat{Y} := \boldsymbol{b}^T C^{-1} \boldsymbol{X}.
\end{equation}
\end{definition}

\begin{definition}[Mean Squared Estimation Error] \label{def:mse}
Given $X_1, \ldots, X_n$, the linear least squares estimator of $Y$ results in a mean squared estimation error of
\begin{equation}
\mathbb{E}\left[\left(Y - \hat{Y}\right)^2\right] := \text{Var}(Y) - \boldsymbol{b}^T C^{-1} \boldsymbol{b}.
\end{equation}
\end{definition}

\section{Problem Formulation} \label{sec:prob_form}
 Let $D \subset \mathbb{R}^d$ represent the environment and let $\sigma_0 \in \mathbb{R}_{> 0}$ be a positive real number. For any location $x \in D$, let $Z(x)$ be a random variable with zero mean and variance $\sigma_0^2$. We consider a convex set of measurement locations $\Theta \subset D$ and a finite set of prediction locations $\Omega \subset \Theta$. Given a positive integer $k \in \mathbb{Z}_{+}$, our goal is to minimize the mean-squared error of the linear estimation of the prediction variables $\left\{ Z(x): x \in \Omega\right\}$ using only $k$ measurement variables.
 
 \begin{rem}
 Note that when $|\Omega| < k$, measurements at all prediction locations will yield low estimation error. The problem is only interesting when $|\Omega| > k$. 
 \end{rem}

We define $\phi_\text{SE}: \mathbb{R}_{\geq 0} \rightarrow \mathbb{R}_{> 0}$ to be the squared exponential covariance function with known parameters $\sigma_0$ and $L \in \mathbb{R}_{>0}$:
\begin{equation} \label{eqn:exp_cov}
    \phi_\text{SE}(x) = \sigma_0^2 e^{-\frac{x^2}{2L^2}},
\end{equation}
The parameters can be learned from a pilot deployment or expert knowledge and is a standard assumption in sensor placement algorithms \cite{krause2008near}.

For any $x, y \in D$, we assume the covariance of the random variables $Z(x), Z(y)$ is given by
\begin{equation}
    \text{Cov} (Z(x), Z(y)) = \mathbb{E}[Z(x)Z(y)] = \phi_{\text{SE}}(\lVert x - y \rVert).
\end{equation}
Let $i$ be a positive integer. For any $x \in D$, let $Y_i(x)$ be the $i^{\text{th}}$ noisy measurement of $Z(x)$ and let the associated noise be $\epsilon_i(x)$. The noise is assumed to be a zero mean random variable with variance $\sigma^2 > 0$. In addition, the noise is uncorrelated across measurements and locations i.e.\ for any $x,y \in D$ and for any positive integers $m, n \in \mathbb{Z}_+$, $\text{Cov}(\epsilon_m(x), \epsilon_n(y)) = 0$. The measurement is
\begin{equation}
  Y_i(x) = Z(x) + \epsilon_i(x).
\end{equation}

In order to reduce notational clutter, for any $x \in D$, we drop the subscript $i$ from the measurement variable $Y_i(x)$ and the associated noise $\epsilon_i(x)$. We proceed with the understanding that if there are multiple measurements at the same location, the associated noise terms are uncorrelated. In addition, any measurement at the same location $x \in D$ (even if there are multiple) will be denoted by $Y(x)$. Now, the measurement equation is
\begin{equation}
  Y(x) = Z(x) + \epsilon(x).
\end{equation}

We wish to minimize the total mean-squared error, which gives us the following constrained optimization problem:
\begin{equation} \label{eqn:initial_mse}
  \underset{S \subset \Theta, |S| \leq k}{\min}\ \underset{x \in \Omega}{\sum}\ \mathbb{E} \left[\left(Z(x) - \hat{Z}(x,S)\right)^2\right],
\end{equation}
where $\hat{Z}(x,S)$ is the linear estimator of $Z(x)$ given the variables in $S$. We will now rewrite the problem using the definition of the mean squared error. Denote the elements of a set $S$ by $\{x_1, \ldots, x_k\}$. The linear estimator $\hat{Z}(x,S)$ is given by Definition \ref{def:llse}:
\begin{equation}
    \hat{Z}(x,S) := \boldsymbol{b}_x(S)^T C(S)^{-1} \boldsymbol{Y}_S
\end{equation}
where
\begin{equation}  \label{eqn:cov_defns}
    \begin{split}
        \boldsymbol{b}_x(S) &:= [\phi_\text{SE}(\lVert x-x_1 \rVert), \ldots, \phi_\text{SE}(\lVert x-x_k \rVert)] \in \mathbb{R}^k\\
        \boldsymbol{Y}_S &:= [Y(x_1), \ldots, Y(x_k)] \in \mathbb{R}^k\\
        C(S) &:= \mathbb{E} \left[ \boldsymbol{Z}_S \boldsymbol{Z}_S^T\right] + \sigma^2 I_k \in \mathbb{R}^{k \times k}\\
        &= \begin{bmatrix}
        \phi_\text{SE}(0) & \dots & \phi_\text{SE}(\lVert x_1 - x_k \rVert)\\
        \vdots & \ddots & \vdots\\
        \phi_\text{SE}(\lVert x_k - x_1 \rVert) & \dots & \phi_\text{SE}(0)
        \end{bmatrix} \\
        &+ \sigma^2 I_k.
    \end{split}
\end{equation}

Using Definition \ref{def:mse}, we can rewrite Equation (\ref{eqn:initial_mse}) as
\begin{equation}
  \underset{S \subset \Theta, |S| \leq k}{\min}\ \underset{x \in \Omega}{\sum}\ \phi_\text{SE}(0) - \boldsymbol{b}_x(S)^T C(S)^{-1} \boldsymbol{b}_x(S).
\end{equation}

Since $\phi_\text{SE}(0) = \sigma_0^2$ is a constant, we can consider the maximization version of the problem. Define
\begin{equation} \label{eqn:obj_defns}
    \begin{split}
        f_x(S) &:= \boldsymbol{b}_x(S)^T C(S)^{-1} \boldsymbol{b}_x(S)\\
        f(S) &:= \sum_{x \in \Omega} f_x(S),
    \end{split}
\end{equation}
where $\boldsymbol{b}_x(S)$ and $C(S)$ are defined in (\ref{eqn:cov_defns}). The function $f_x(S)$ is also known as the \textit{squared multiple correlation} \cite{miller2002subset,das2008algorithms} or the \textit{variance reduction} \cite{krause2008robust}.

In this paper, we wish to find the measurement set that maximizes the total variance reduction. This is formulated as the following optimization problem.
\begin{problem}
Given measurement locations $\Theta$, prediction locations $\Omega$, and a budget $k > 0$, find a measurement set $S \subset \Theta$ of size $k$ that maximizes the total variance reduction:
\begin{equation} \label{eqn:opt_prob}
  \underset{S \subset \Theta, |S| \leq k}{\max}\ f(S).
\end{equation}
\end{problem}

\section{Problem Structure} \label{sec:prob_struct}

In this section we provide preliminary results that guide the design of our algorithm, presented in the next section.

\subsection{Non-submodularity} \label{subsec:nonsubmod}

Problem 1 resembles a sensor placement problem where one is interested in a subset of locations to deploy sensors to maximize the information gained about the environment. Metrics related to the information gained such as coverage and mutual information are known to be submodular functions which can be approximately solved efficiently with a guarantee. However, for Problem 1, we provide an example to show the variance reduction objective is \textit{not submodular}. 
\begin{exmp}
Consider the following environment setup where the points lie on an interval on the real line.
\begin{equation}
\begin{split}
    D \subset \mathbb{R}, \Omega = \{0\}, \Theta = [0, 2], \sigma = 1, \sigma_0 = 1, L = 1, \\
    A = \{0.6784\}, B = \{0.6784, 1.4869\}, x = 0.6892.
\end{split}
\end{equation}
Now, $f(A \cup \{x\}) - f(A) = 0.1021$ and $f(B \cup \{x\}) - f(B) = 0.1025$, which shows the violation.
\hfill\qedsymbol
\end{exmp}

\subsection{Two Prediction Locations with One Sample}
In this subsection, we discuss properties of the problem in 1-D, i.e.,\ the random variables are associated with locations on the real line. This restriction provides valuable insight into the problem and motivates our proposed algorithm.

Suppose the set of prediction locations contains two points i.e.\ $\Omega = \{y_1, y_2 \} \subset [a,b]$, with $y_1 < y_2$. After some simplification, the optimization problem in (\ref{eqn:opt_prob}) is
\begin{equation} \label{eqn:2test_loc_1sample}
\begin{split}
  &\underset{x \in [a,b]}{\max}\ \frac{1}{\sigma_0^2 + \sigma^2} \left(\phi_\text{SE}^2(\lVert x-y_1 \rVert) + \phi_\text{SE}^2(\lVert x-y_2 \rVert)\right) \\
  = &\underset{x \in [a,b]}{\max}\ \frac{\sigma_0^4}{\sigma_0^2 + \sigma^2} \left(e^{-\frac{1}{L^2} \lVert x-y_1 \rVert^2} + e^{-\frac{1}{L^2} \lVert x-y_2 \rVert^2}\right).
 \end{split}
\end{equation}
The solution depends on the relationship between the distance between the two prediction locations and $L$, the parameter of the squared exponential covariance function defined in \eqref{eqn:exp_cov}. This is formalized in the following proposition.

\begin{figure}
    \centering
    \includegraphics[width=0.48\textwidth]{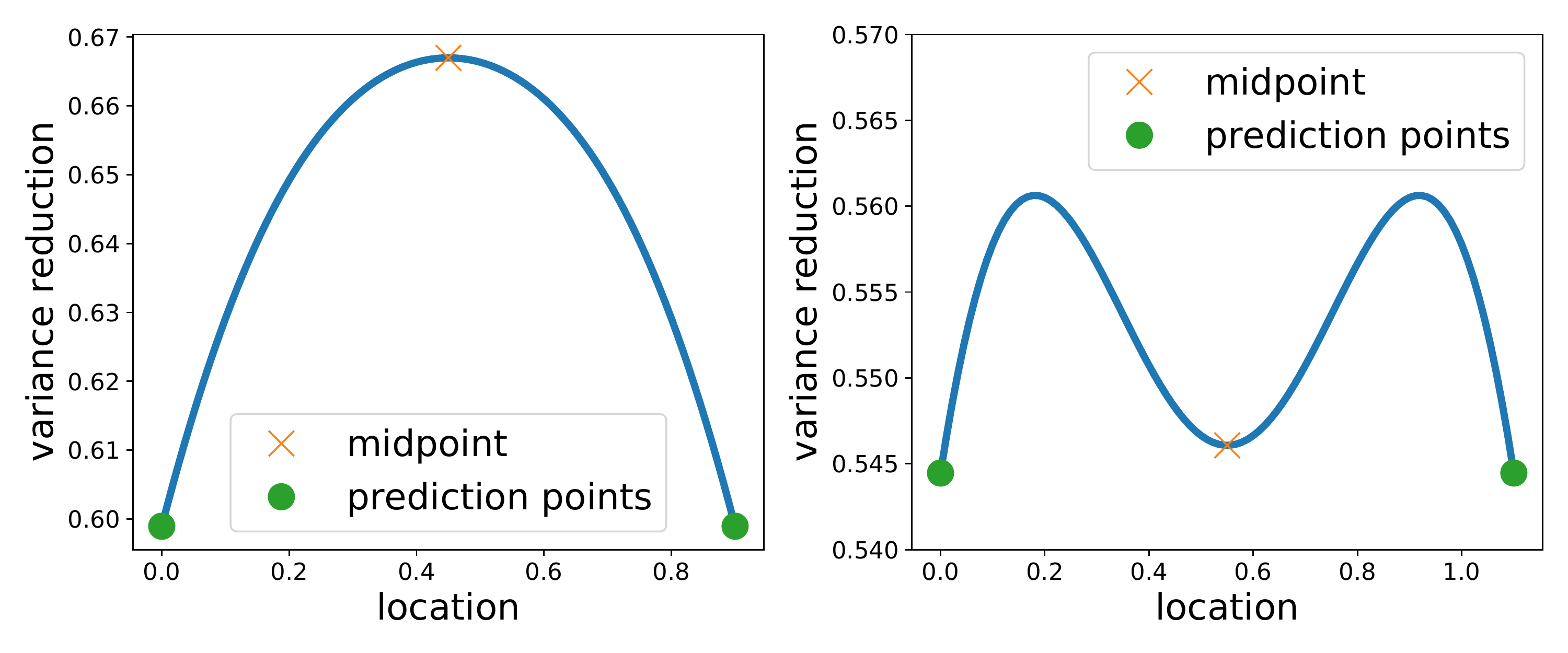}
    \caption{Left: An example of the objective function defined on the interval $[0,1]$. The test locations are located at $y_1 = 0.0, y_2 = 0.9$ and $L = \frac{1}{\sqrt{2}}$. In this setting, the midpoint achieves the global maximum. Right: An example of the objective function defined on the interval $[0,1.1]$. The test locations are located at $y_1 = 0.0, y_2 = 1.1$ and $L = \frac{1}{\sqrt{2}}$. In this setting, the midpoint is a local minima.}
    \label{fig:optwL}
\end{figure}

\begin{proposition} \label{prop:1sample_2loc}
    Let $D \subset \mathbb{R}$, $\Omega = \{y_1, y_2\} \subset D, \Theta = D$, $t=1$, and the midpoint $x^* = \frac{y_1 + y_2}{2}$. Denote the optimal solution to (\ref{eqn:2test_loc_1sample}) by $\textsc{OPT}$. Then,
    \begin{equation}
    \textsc{OPT} = x^* \iff \lVert y_2 - y_1 \rVert \leq \sqrt{2} L.
    \end{equation}
\end{proposition}
\begin{proof}
Define $d_1 := \lVert x - y_1 \rVert, d_2 := \lVert x - y_2 \rVert$. In this setting, the objective function is
\begin{equation}
    f(x) = \frac{\sigma_0^4}{\sigma_0^2 + \sigma^2} \left( e^{-\frac{d_1^2}{L^2}} + e^{-\frac{d_2^2}{L^2}} \right).
\end{equation}
The derivative of $f(x)$ is:
\begin{equation} \label{eqn:derivative}
    f'(x) = \frac{-2}{L^2} \frac{\sigma_0^4}{\sigma_0^2 + \sigma^2} \left( (x-y_1) e^{-\frac{d_1^2}{L^2}} + (x-y_2) e^{-\frac{d_2^2}{L}} \right).
\end{equation}
In general, it is difficult to solve for the critical points using the derivative of the form in (\ref{eqn:derivative}) since it is a transcendental equation. One could resort to numerical methods to solve it. However, in this case, we identify the midpoint $x^* = \frac{y_1 + y_2}{2}$ as a critical point for the function i.e.\ $f'(x^*) = 0$.

The second derivative is given by:
\begin{equation}
\begin{split}
    f''(x) = \frac{-2}{L^2} \frac{\sigma_0^4}{\sigma_0^2 + \sigma^2} \bigg(e^{-\frac{d_1^2}{L^2}} \big( 1-\frac{2}{L^2} (x-y_1)^2 \big) \\
    + e^{-\frac{d_2^2}{L^2}} \big(1-\frac{2}{L^2} (x-y_2)^2\big)\bigg)
\end{split}
\end{equation}

Evaluating the second derivative at the critical point $x^*=\frac{y_1+y_2}{2}$ gives
\begin{equation}
\begin{split}
f''(x^*) = \frac{-4}{L^2} \frac{\sigma_0^4}{\sigma_0^2 + \sigma^2} e^{-\frac{1}{4L^2} (y_2-y_1)^2} 
\\
\left( 1 - \frac{1}{2L^2} ( y_2 - y_1 )^2\right)
\end{split}
\end{equation}

$\Rightarrow$
We will prove the forward direction by proving the contrapositive. When $(y_2 - y_1)^2 > 2L^2$, $f''(x^*) > 0$ and $x^*$ is a local minima and is not optimal.

$\Leftarrow$
Since $(y_2 - y_1)^2 < 2L^2$, we have that $f''(x^*) < 0$, and thus $x^*$ is a local maxima. 
To show $x^*$ is a global maximum, we show $f'(x) > 0$ on the interval $[y_1, x^*]$ and $f'(x) < 0$ on the interval $[x^*, y_2]$.

On the interval $[y_1, x^*]$, it is sufficient to show the following:
\begin{equation} \label{eqn:proof1}
    \frac{x-y_1}{y_2 - x} \leq \frac{e^{-\frac{d_2^2}{L^2} }}{e^{-\frac{d_1^2}{L^2} }},
\end{equation}
since this ensures $f'(x) > 0$.
Consider the RHS in (\ref{eqn:proof1}),
\begin{equation} \label{eqn:proof1_2nd}
    \begin{split}
    \frac{e^{-\frac{d_2^2}{L^2} }}{e^{-\frac{d_1^2}{L^2} }} &= e^{-\frac{1}{L^2} \left( (x-y_2)^2 - (x-y_1)^2 \right)}\\
    &= e^{-\frac{1}{L^2}(y_1-y_2)(2x-y_2-y_1)}
    = e^{-\frac{2}{L^2}(y_2-y_1)(\frac{y_1+y_2}{2} - x)}.
    \end{split}
\end{equation}
Since $(y_2-y_1)^2 < 2L^2$, it holds that $-\frac{2}{L^2} (y_2-y_1) > -\frac{4}{(y_2-y_1)}$. Continuing from (\ref{eqn:proof1_2nd}) gives
\begin{equation} \label{proof1:lb}
    \begin{split}
    e^{-\frac{2}{L^2}(y_2-y_1)(\frac{y_1+y_2}{2} - x)} &> e^{-\frac{4}{y_2-y_1}(\frac{y_1+y_2}{2} - x)}.
    \end{split}
\end{equation}
Now, we need to show that the LHS in (\ref{eqn:proof1}) is less than the lower bound in (\ref{proof1:lb}). Define $A:= \frac{y_1+y_2}{2}$, $B:=\frac{y_2-y_1}{2}$, and for any $x \in [y_1, x^*]$, define $Z := x - A$. Starting with the LHS gives
\begin{equation}
    \begin{split}
      \frac{x-y_1}{y_2 - x} &= \frac{x-y_1 + A - A}{y_2 - x + A - A}
      = \frac{x+B-A}{-x + A + B}\\
      &= \frac{Z + B}{-Z + B} = \frac{\frac{Z}{B} + 1}{-\frac{Z}{B} + 1}.
    \end{split}
\end{equation}
Consider the function $g(y) := e^{2y} \frac{1-y}{1+y}$. The derivative is $g'(y) = - \frac{2y^2 e^{2y}}{(1+y)^2}$ which shows the function is non-increasing for all $y \neq -1$. Then, since $g(0) = 1$, $g(y) \geq 1$ on the interval $(-1, 0]$. Then, we have for $y \in (-1, 0]$,
\begin{equation} \label{proof1:fn}
    \frac{y+1}{-y+1} \leq e^{2y}.
\end{equation}
Note that since $y_2 - y_1 \leq \sqrt{2}L$, for any $x \in [y_1, x^*]$, $-1 \leq \frac{Z}{B} \leq 0$. Setting $y = \frac{Z}{B}$ in (\ref{proof1:fn}) gives
\begin{equation}
      \frac{\frac{Z}{B} + 1}{-\frac{Z}{B} + 1} \leq e^{2 \frac{Z}{B}} = e^{-\frac{4}{y_2-y_1}(\frac{y_1+y_2}{2} - x)},
\end{equation}
which shows that the LHS in (\ref{eqn:proof1}) is less than the lower bound in (\ref{proof1:lb}). Thus, $f'(x) > 0$ and $f(x)$ is increasing on the interval $[y_1, x^*]$. Using similar arguments, one can show $f(x)$ is decreasing on the interval $[x^*, y_2]$. Combining this with the fact $x^*$ is a critical point and $f''(x^*) < 0$ implies $x^*$ is the global maximum.
\end{proof}

When the points are separated by a distance greater than $\sqrt{2} L$, Proposition \ref{prop:1sample_2loc} guarantees the suboptimality of the midpoint (see Figure~\ref{fig:optwL}). In this case, the prediction locations are reasonable solutions whose performance guarantee is given by the following proposition.

\begin{proposition} \label{prop:pred_greater_dist}
   Given $D \subset \mathbb{R}$, $\Omega = \{y_1, y_2\} \subset D, \Theta = D,$ and $k=1$, when $\lVert y_2 - y_1 \rVert > \sqrt{2} L$, the point $x=y_1$ is an approximate maximizer to (\ref{eqn:2test_loc_1sample}) with a guarantee
   \begin{equation}
       \frac{f(\{y_1\})}{f(\{x^*\})} \geq 0.62,
   \end{equation}
   where $x^*$ is the optimal measurement location.
\end{proposition}
\begin{proof}
Define $d_1 := \lVert x - y_1 \rVert, d_2 := \lVert x - y_2 \rVert$. In this setting, the objective function is
\begin{equation}
    f(x) = \frac{\sigma_0^4}{\sigma_0^2 + \sigma^2} \left( e^{-\frac{d_1^2}{L^2}} + e^{-\frac{d_2^2}{L^2}} \right).
\end{equation}
The derivative of $f(x)$ is:
\begin{equation} 
    f'(x) = \frac{-2}{L^2} \frac{\sigma_0^4}{\sigma_0^2 + \sigma^2} \left( (x-y_1) e^{-\frac{d_1^2}{L^2}} + (x-y_2) e^{-\frac{d_2^2}{L}} \right).
\end{equation}
For $x < y_1$, $f'(x)$ is positive and for $x > y_2$, $f'(x)$ is negative. Thus, the optimal solution $x^*$ must lie within the interval $[y_1, y_2]$. Since $\lVert y_1 - y_2 \rVert > \sqrt{2} L $, $x = \frac{y_1 + y_2}{2}$ is a local minima (Proposition \ref{prop:1sample_2loc}). The function is symmetric around the midpoint, so we restrict our discussion to the interval $[y_1, \frac{y_1+y_2}{2}]$.
A lower bound for the solution $x = y_1$ is constructed by assuming $e^{-\frac{\lVert y_1 - y_2 \rVert^2}{L^2}} = 0$. Thus, $f(y_1) \geq \frac{\sigma_0^2}{\sigma_0^2 + \sigma^2}$. Since the optimal solution $x^* \in [y_1, \frac{y_1+y_2}{2}]$ and $\lVert y_1 - y_2 \rVert \geq \sqrt{2} L$, an upper bound can be constructed: $f(x^*) < \frac{\sigma_0^2(1 + e^{-0.5})}{\sigma_0^2 + \sigma^2}$. Thus,
\begin{equation}
    \frac{f(y_1)}{f(x^*)} \geq \frac{1}{1+e^{-0.5}} \approx 0.62.
\end{equation}
\end{proof}
Propositions \ref{prop:1sample_2loc} and \ref{prop:pred_greater_dist} motivate our algorithm design. For two prediction points and one sample in 1D, either the midpoint is optimal or either prediction point is an approximate solution. This suggests the following idea: restrict the search of the greedy algorithm to the prediction locations and the centroids of the cliques formed by the prediction locations.

\begin{figure}[t]
    \centering
    \includegraphics[width=0.48\textwidth]{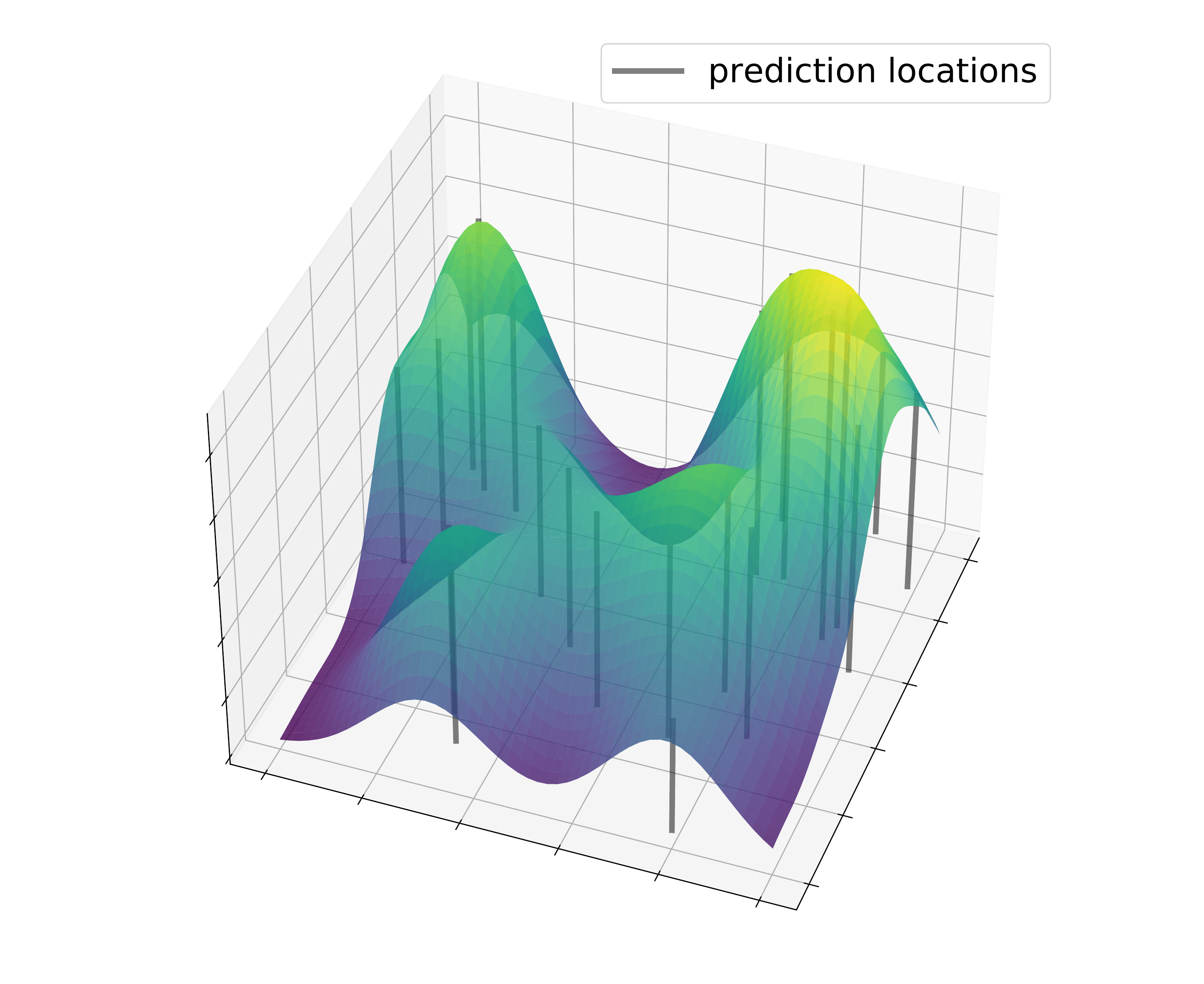}
    \caption{A plot of the objective $f(S)$ in Problem 1 when the random variables are associated with a two-dimensional space and the budget is one i.e.\ $t = 1$. The function is non-concave and has many maxima.}
    \label{fig:nonconcave_obj}
\end{figure}
\section{Algorithm} \label{sec:alg}
In this section, we discuss \ggr\ and its limitations, our proposed algorithm \igr\ based on computing centroids of maximal cliques, and provide a reformulation of computing the marginal gains that speeds up the implementation of both greedy algorithms in practice.

\subsection{\ggr} \label{subsec:complexity_greedy}
The greedy algorithm is popular for subset selection in regression where it is also known as Forward Selection \cite{das2008algorithms,miller2002subset,hastie01statisticallearning}. In this section, we discuss how the greedy algorithm can be used for infinite observation sets. For Problem 1, starting with $S_0 = \emptyset$, the first step of the algorithm computes the maximizer to
\begin{equation}
    \begin{split}
        S_{1} &= \underset{x \in \Theta}{\text{arg max}}\ f(\{x\})\\
        &= \underset{x \in \Theta}{\text{arg max}}\ \frac{1}{\sigma^2 + \sigma_0^2} \sum_{y \in \Omega} \phi_\text{SE}^2(\lVert x - y \rVert)\\
        &= \underset{x \in \Theta}{\text{arg max}}\ \frac{\sigma_0^4}{\sigma^2 + \sigma_0^2} \sum_{y \in \Omega} e^{-\frac{1}{L^2} \lVert x - y \rVert^2}.
    \end{split}
\end{equation}
This function is non-concave and in general, it is difficult to find the global maximum. A plot of this objective when the set of observation locations is a subset of the two dimensional Euclidean space i.e.\ $\Theta \subset \mathbb{R}^2$ is shown in Figure~\ref{fig:nonconcave_obj}.

To tackle this non-concave maximization problem, the set of measurement locations $\Theta$ can be uniformly discretized to form a finite set of points $\bar{\Theta} \subset \Theta$. The point $x \in \bar{\Theta}$ with the maximum function value is returned as an approximate solution. This is known as the Uniform Grid method \cite{nesterov2018lectures}. Each step of the greedy algorithm can be approximately solved using this method. The grid discretization is determined by a positive integer parameter $\rho \geq 1$ which tiles each dimension with $\rho$ points to form $\bar{\Theta}$ of size $\rho^d$. We refer to this method as \ggr. The time complexity of \ggr\ is given in the following proposition.

\begin{proposition} \label{prop:greedy_time_complexity}
    Given a positive integer $\rho \geq 1$ and a grid discretization of size $\rho^d$, \ggr\ finds a solution to Problem 1 in time $O(\rho^d k^3\max\{k, |\Omega|\})$.
\end{proposition}
\begin{proof}
\ggr\ runs for $k$ iterations with $\rho^d$ function evaluations per iteration. For a set $S$ of size $k$, the evaluation of $f(S)$ requires the inversion of a $k \times k$ matrix and $|\Omega|$ matrix multiplications, with each multiplication taking time $O(k^2)$. Thus, the overall time complexity for the evaluation of $f(S)$ is $O(\max\{k^3, |\Omega| k^2\})$. Then, \ggr\ runs in time $O(\rho^d k \max\{k^3, |\Omega| k^2 \}) = O(\rho^d k^3 \max\{k, |\Omega|\})$.
\end{proof}

The dependence of the runtime on $\rho^d$ is concerning. To get good quality solutions using the greedy algorithm, $\rho$ needs to be sufficiently large to achieve a good grid resolution. In this paper, we aim to find good quality solutions using the greedy algorithm in time \textit{independent} of the grid discretization.

\subsection{\igr}
In this section, we present our algorithm \igr\ which involves two parts. First, we find the centroids of maximal cliques in a graph with nodes as prediction locations. Second, we use the set of centroids and prediction locations as a ground set for the greedy algorithm for maximizing set functions (Definition \ref{def:greedyalg}) to solve Problem 1.
\begin{algorithm}[t]
	\DontPrintSemicolon 
	\KwIn{Prediction locations $\Omega$}
	\KwOut{Clique Centroids $\mathcal{X} \subset \Theta$}		
	$G=(V,E) \leftarrow \textsc{constructGraph}(\Omega)$\\
	$\mathcal{C} \leftarrow \textsc{maximalCliques}(G)$\\
	Initialize $\mathcal{X} = \emptyset$\\
	\For{each clique $\mathcal{M} \in \mathcal{C}$}{
	$\mathcal{X} \leftarrow \mathcal{X} \cup \left\{\textsc{centroid}(\mathcal{M})\right\}$
	}
	\Return{$\mathcal{X}$}
	\caption{\textsc{MaximalCliqueCentroids}}
	\label{alg:maximal_cliques}
\end{algorithm}

\subsubsection*{Finding Clique Centroids}
The steps to compute clique centroids is given in Algorithm~\ref{alg:maximal_cliques}. The first step (Line 1, $\textsc{constructGraph}$) constructs a graph $G = (V,E)$ with vertices as prediction locations. Two vertices are connected with an edge if the corresponding prediction locations are within a distance $\sqrt{2} L$. An example of a constructed graph for a two dimensional problem is shown in Figure~\ref{fig:fig7}. The next step is to compute the clique centroids. Ideally, we would like to find maximum cliques in the graph. Unfortunately, finding maximum cliques is NP-Hard \cite{cormen2009introduction}. We limit ourselves to finding \textit{maximal} cliques from each vertex in the graph since this can be done efficiently with a greedy algorithm: for each vertex $v \in V$ in the graph, grow the clique one vertex at a time by looping through the remaining vertices, add it to the clique if it is adjacent to every vertex in the clique and discard it otherwise (Line 2, $\textsc{maximalCliques}$). Note, this method does not yield \textit{all} maximal cliques like the Bron-Kerbosch algorithm \cite{bron1973algorithm}, which has an exponential time complexity in the worst case. Once we have the set of maximal cliques, the final step is to loop through the cliques and compute the centroid of the prediction locations associated with the clique (Line 5).

\begin{figure}[t]
    \centering
    \includegraphics[width=0.47\textwidth]{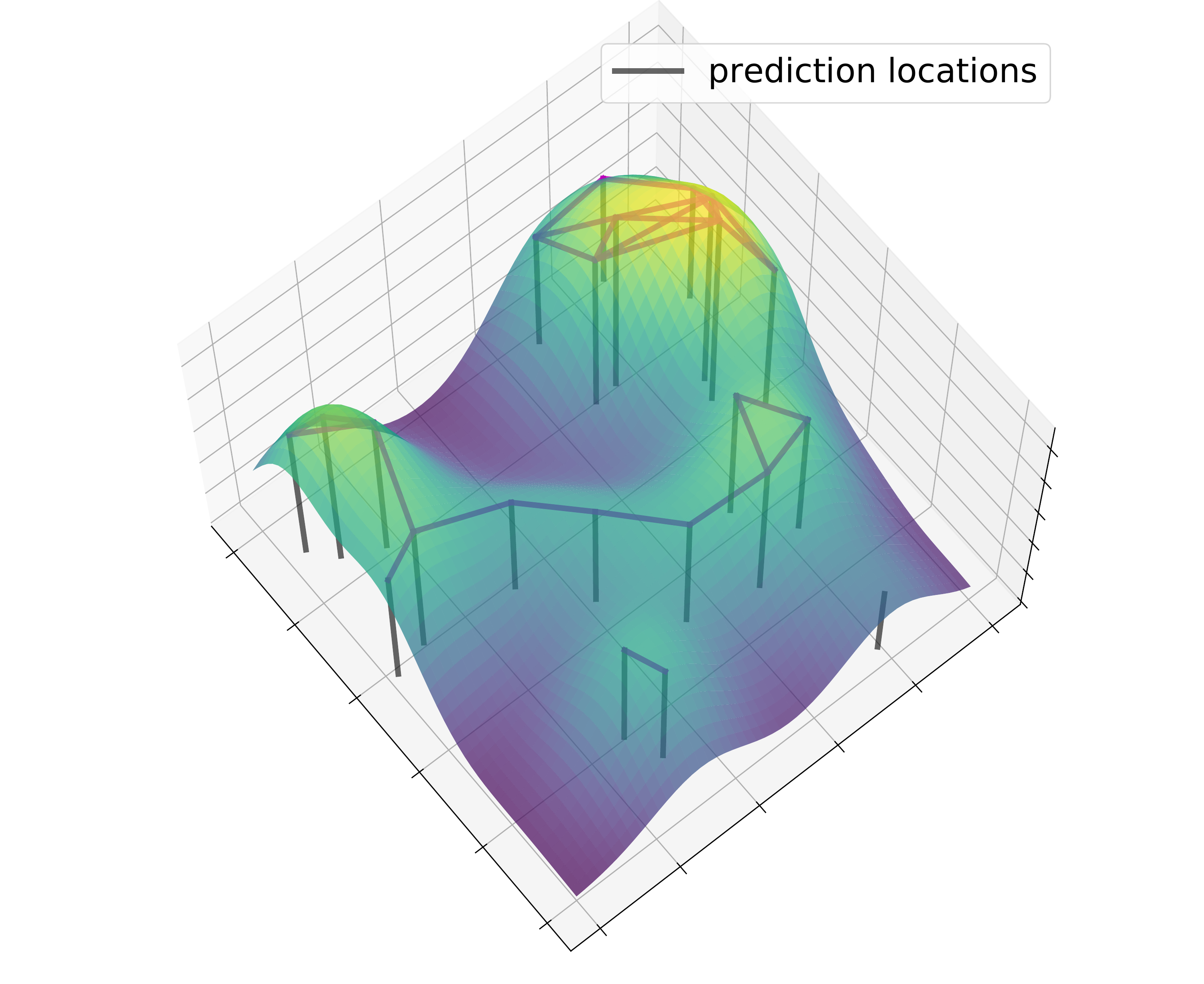}
    \caption{The objective function $f(S)$ when the budget $k=1$ for a given set of prediction locations in two dimensions. Two prediction locations are connected by an edge if the distance between them is less than or equal to $\sqrt{2}L$.}
    \label{fig:fig7}
\end{figure}

\subsubsection*{\igr}
Proposition 3 ensures that the prediction locations are reasonable approximate solutions when the prediction locations are separated by a distance greater than $\sqrt{2} L$. Instead of \ggr\ which has a runtime of $O(\rho^d k^3 \max\{k, |\Omega|\})$ for Problem 1 (see Proposition \ref{prop:greedy_time_complexity}), we remove the dependence on $\rho^d$ i.e.\ the grid discretization, by limiting the search to the set of centroids (computed in Algorithm~\ref{alg:maximal_cliques}) and the set of prediction locations: $\mathcal{X} \cup \Omega$. The set of centroids is a feasible set for Problem 1 since $\Omega \subset \Theta$ and $\Theta$ is a convex set i.e. the set of measurement locations $\Theta$ contains the centroids of any subset of prediction locations. Since the number of maximal cliques computed by Algorithm~\ref{alg:maximal_cliques} is bounded above by the number of prediction locations, the runtime of \igr\ is $O(|\Omega| k^3 \max\{k, |\Omega|\})$. This is an improvement over the runtime $O(\rho^d k^3 \max\{k, |\Omega|\})$ of \ggr, as long as $|\mathcal{X} \cup \Omega| < \rho^d$, which we will show in Section \ref{sec:experiments}, is required for \ggr\ to obtain good solutions for large fields. The steps for \igr\ are given in Algorithm~\ref{alg:lazy_greedy}.
\begin{algorithm}	
	\DontPrintSemicolon 
	\KwIn{Continuous Field: $\Theta$, Prediction Locations: $\Omega$, budget $k > 0$}
	\KwOut{Measurement Set: $S \subset \Theta$, $|S| = k$}		
	$\mathcal{V} = \textsc{MaximalCliqueCentroids}(\Omega)$\\
	Initialize $S_0 = \emptyset$\\
	\For{$i = 1\ \text{to}\ k$}{
    $S_{i} = S_{i-1} \cup \{\underset{x \in \mathcal{V}}{\text{arg max}}\ f(S_{i-1} \cup \{x\}) - f(S_{i-1})\}$
	}
	\Return{$S_k$}
	\caption{\igr}
	\label{alg:lazy_greedy}
\end{algorithm}

\subsection{Implementation of the Greedy Algorithm}
Each step of the greedy algorithm requires computing the maximizer of the marginal gain $f(S_i \cup \{x\}) - f(S_i)$ over all feasible $x$. Computing $f(S)$ in the form in (\ref{eqn:obj_defns}) is time consuming and is not amenable to vectorization in NumPy \cite{harris2020array} directly. Using Proposition \ref{prop:marginal}, the marginal can be rewritten in a form that can be vectorized, and in practice is much faster to compute. For example, computing the solution for 500 prediction points, ground set size of $400$, and a budget of $25$ takes $\approx 0.5$ seconds with vectorization and $\approx 14$ seconds with the non-vectorized version.

\begin{proposition} \label{prop:marginal} The marginal improvement of $f(S)$ when adding an element $x$ to a set $A$ is given by:
\begin{equation}
f(A \cup \left\{x\right\}) - f(A) = T_x \sum_{y \in \Omega} \left(R_{x,y}\left(A\right) - \phi_{\text{SE}}^2(x-y)\right)^2,
\end{equation}
where $T_x = \left(\sigma_0^2 + \sigma^2 - \boldsymbol{b}_{x}^T\left(A\right) C\left(A\right)^{-1} \boldsymbol{b}_{x}\left(A\right) \right)^{-1}$ and $R_{x,y}\left(A\right) = \boldsymbol{b}_{x}(A)^T C\left(A\right)^{-1} \boldsymbol{b}_{y}\left(A\right)$.
\end{proposition}
\begin{proof}
We can partition the covariance matrix $C(A \cup \{x\})$ as follows:
\begin{equation}
    C(A \cup \{x\}) = \begin{bmatrix} 
    C(A) & \boldsymbol{b}_x(A)\\
    \boldsymbol{b}_x(A)^T & \sigma_0^2 + \sigma^2
    \end{bmatrix}
\end{equation}

Define $T_x:= \sigma_0^2 + \sigma^2 - \boldsymbol{b}_x(A)^T C(A)^{-1} \boldsymbol{b}_x(A)$. Using block matrix inversion,
\begin{equation}
    \begin{split}
    C&(A \cup \{x\}) = \\
    &\left[\begin{smallmatrix}
        C(A)^{-1} + C(A)^{-1} \boldsymbol{b}_x(A) T_x \boldsymbol{b}_x(A)^T C(A)^{-1} & -C(A)^{-1} \boldsymbol{b}_x(A) T_x\\
        -T_x \boldsymbol{b}_x(A)^T C(A)^{-1} & T_x
    \end{smallmatrix}\right]
    \end{split}
\end{equation}

Consider the objective
\begin{equation}
    f(A \cup \{x\}) = \sum_{y \in \Omega} \boldsymbol{b}_y(A \cup \{x\}) C(A \cup \{x\})^{-1} \boldsymbol{b}_y(A \cup \{x\}).
\end{equation}

We can partition $\boldsymbol{b}_y(A \cup \{x\})$ as follows:
\begin{equation}
    \boldsymbol{b}_y(A \cup \{x\}) = \begin{bmatrix}
        \boldsymbol{b}_y(A)\\
        \phi_\text{SE}(x-y)
    \end{bmatrix}
\end{equation}

Define $R_{x,y}\left(A\right) = \boldsymbol{b}_{x}(A)^T C\left(A\right)^{-1} \boldsymbol{b}_{y}\left(A\right)$. Then, plugging in the required quantities in the objective and performing the vector-matrix multiplications results in
\begin{equation}
    \begin{split}
        f(A \cup \{x\}) &= f(A) +\sum_{y \in \Omega} \Big(T_x R_{x,y}(A)^T R_{x,y}(A) \\
        &- 2\phi_\text{SE}(x-y) T_x R_{x,y}(A)^T \\
        &+ \phi^2_\text{SE}(x-y) T_x \Big)\\
        f(A \cup \{x\}) - f(A) &= T_x \sum_{y \in \Omega} \left( R_{x,y}(A) - \phi_\text{SE}(x-y)\right)^2.
    \end{split}
\end{equation}
\end{proof}
\section{Evaluation} \label{sec:experiments}
In this section, we provide evidence of two advantages of \igr\ over \ggr. First, \igr\ obtains higher quality solutions on problem instances where the run time of both algorithms is comparable. Second, on instances where the solution quality is comparable, \igr\ finds the solution faster than \ggr. The solution quality is measured by the mean squared error (Equation \ref{eqn:initial_mse}) and the run time is measured in seconds.

\subsubsection*{Experimental Setup}
We follow the setup in \cite{suryan2020learning} where a Gaussian Process was fit to a real world dataset of organic matter measurements in an agricultural field \cite{mulla2000evaluation}. Note that we do not require the actual data, only the parameters of the squared exponential covariance function and the variance of the measurement noise. Specifically, the authors \cite{suryan2020learning} computed $L = 8.33$ meters, $\sigma_0 = 12.87$, and $\sigma^2 = 0.0361$. The interpretation of $L$ is the distance one has to travel before the underlying function value changes \cite{rasmussen2003gaussian}. Since the covariance function is a squared exponential, only the relative distances between points matter, not the absolute positions. This enables us to consider different environment sizes:
\begin{enumerate}
\item $D_{\text{small}} := \left\{(x,y) \in \mathbb{R}^2: 0 \leq x \leq 40, 0 \leq y \leq 40\right\}$, $\text{Area} = 1600$ square meters.
\item $D_{\text{med}} := \left\{(x,y) \in \mathbb{R}^2: 0 \leq x \leq 120, 0 \leq y \leq 120\right\}$, $\text{Area} = 14,400$ square meters.
\item $D_{\text{large}} := \left\{(x,y) \in \mathbb{R}^2: 0 \leq x \leq 600, 0 \leq y \leq 600\right\}$, $\text{Area} = 360,000$ square meters.
\end{enumerate}
We also consider three regimes for the number of prediction points: sparse (20 points, budget 8), moderate (300 points, budget 75), and dense (1000 points, budget 200). The results in the following sections are based on 10 randomly generated problem instances for each combination of environment type (small, medium, large) and prediction point regime (sparse, moderate, dense). The experiments are implemented using NumPy \cite{harris2020array} on an AMD Ryzen 7 2700 processor.

\begin{figure}[t]
    \centering
    \includegraphics[width=0.48\textwidth]{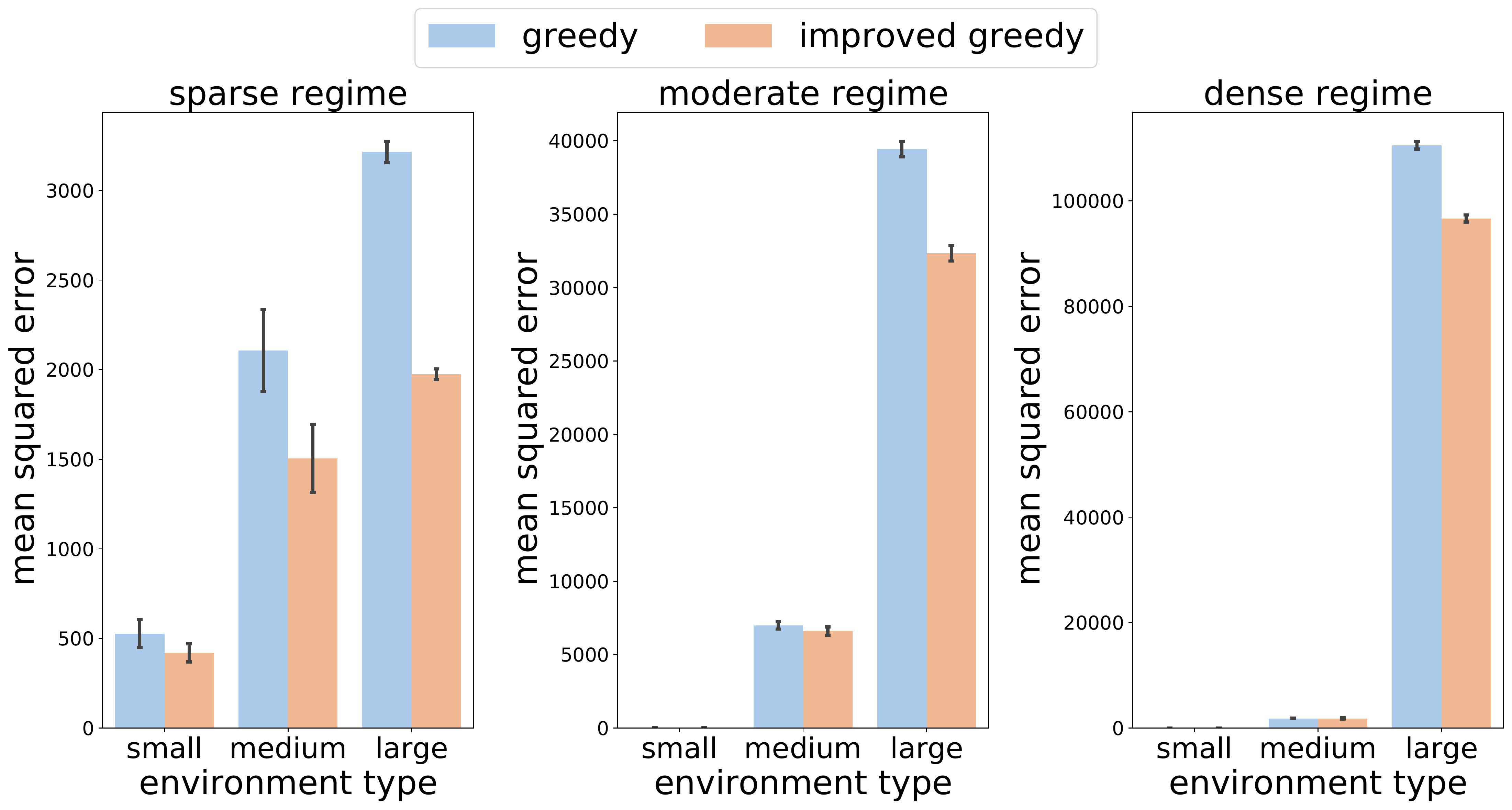}
    \caption{Comparison of the solution quality while keeping the run time approximately the same. \igr\ obtains equal or better solutions \ggr\ in all environment types and regime of prediction points.}
    \label{fig:fig14}
\end{figure}

\begin{figure}[t]
    \centering
    \includegraphics[width=0.48\textwidth]{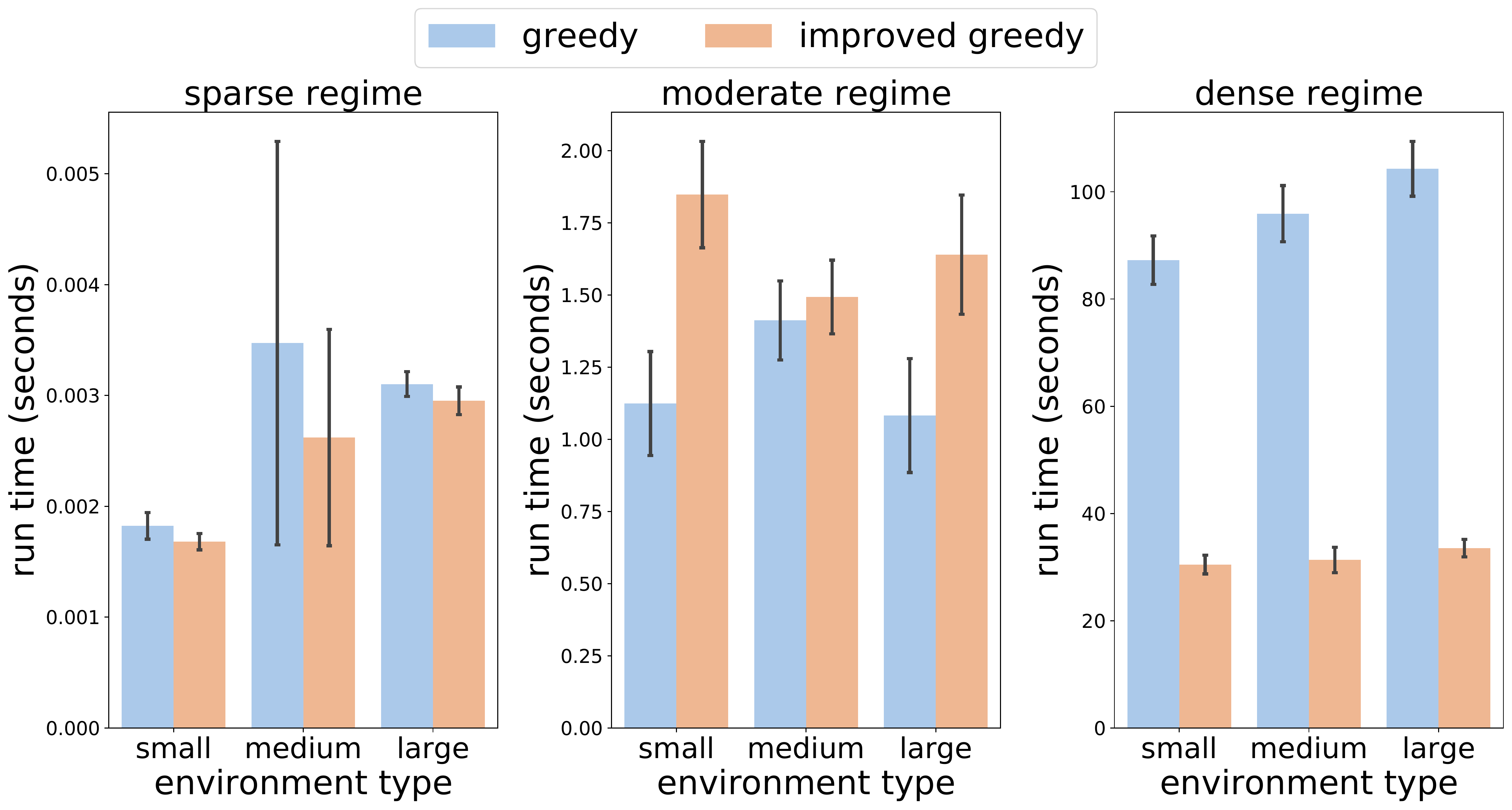}
    \caption{Comparison of the run time while keeping the solution quality approximately the same. \ggr\ practically takes at least as much time as \igr\ to find solutions of similar quality.}
    \label{fig:fig15}
\end{figure}

\subsection{Solution Quality} \label{subsec:exp1}
In the first set of experiments, we aim to answer the following question: given the same computational resources, which algorithm provides a better solution? To ensure equal computational resources, for a $N \times N$ grid discretization, we set $N = \left\lceil \sqrt{2 |\Omega|} \right\rceil$. Since the number of maximal cliques computed in Algorithm~\ref{alg:maximal_cliques} is at most $|\Omega|$, this ensures the runtimes are comparable. The grids selected are: $7 \times 7$ (sparse regime), $25 \times 25$ (moderate regime), $45 \times 45$ (dense regime).

The results are shown in Figure~\ref{fig:fig14}. In the sparse regime (left plot) \igr\ outperforms \ggr\ on average in all environment types. The difference in performance is the highest in large environments since the grid resolution is not sufficient to cover the space. In the moderate regime (center plot) and dense regime (right plot), the solution quality of both algorithms is similar in small and medium sized environments. However, for large environments, \igr\ obtains better solutions. The difference in performance reduces as we move from the sparse to dense regime. This is because the high density of prediction points increases the chance of close proximity with grid points, even in the case of low resolution grids. In summary, using a similar amount of computational resources, \igr\ obtains solutions of equal or higher quality than \ggr\ across all environment sizes and regimes on the number of prediction points.

\subsection{Run Time} \label{subsec:exp2}
In the second set of experiments, we aim to answer the following question: how much longer does it take for \ggr\ to achieve similar solution quality as \igr? For each problem instance, if \igr\ obtains a higher objective value than \ggr\, we repeatedly increase the grid resolution until \ggr\ attains the a similar objective value. We compare the time taken by \ggr\ on the final grid resolution to the time taken by \igr.

The results are shown in Figure~\ref{fig:fig15}. In the sparse regime (left plot) and moderate regime (center plot), the run times are practically the same. In the moderate regime, the number of prediction points is a bit higher than the number of grid points which is the reason for the slightly higher  runtime of \igr. However, in the dense regime, \ggr\ takes approximately 2.5 times (small environments), 4 times (medium size environments), and 5 times (large environments) as long as \igr\ to attain a similar objective value. The runtime of \ggr\ increases with the size of the environment, while the runtime of \igr\ remains fairly constant. Note, while the run times in our experiments seem feasible in practice, the fields considered in these experiments are small compared to average farm sizes. For example, in 2019, the average farm size in USA was 444 acres \cite{usda2019} which is 5 times the size of the largest environment considered here. We expect larger reductions in run time for these agricultural fields in practice. In summary, \igr\ finds solutions of similar quality to \ggr\ more efficiently i.e.\ using less or equal amounts of time, across all environments and regimes on the number of prediction points.

\section{Conclusions}
We discussed the problem of selecting a $k$-subset that yields the best linear estimate at a set of prediction locations in a continuous spatial field. One approach is to solve the problem using a grid discretization of the field and greedily select $k$ measurement locations. However, this can be computationally expensive for large fields. Instead, we restricted the search of the greedy algorithm to the set of prediction locations and the centroids of their cliques. This was motivated by identifying a critical distance between two prediction points which characterized the optimal solution in 1D. In simulations, we showed the effectiveness of the proposed approach in terms of solution quality and runtime.


%%%%%%%%%%%%%%%%%%%%%%%%%%%%%%%%%%%%%%%%%%%%%%%%%%%%%%%%%%%%%%%%%%%%%%%%%%%%%%%%

%%%%%%%%%%%%%%%%%%%%%%%%%%%%%%%%%%%%%%%%%%%%%%%%%%%%%%%%%%%%%%%%%%%%%%%%%%%%%%%%

%%%%%%%%%%%%%%%%%%%%%%%%%%%%%%%%%%%%%%%%%%%%%%%%%%%%%%%%%%%%%%%%%%%%%%%%%%%%%%%%

%%%%%%%%%%%%%%%%%%%%%%%%%%%%%%%%%%%%%%%%%%%%%%%%%%%%%%%%%%%%%%%%%%%%%%%%%%%%%%%%

\Urlmuskip=0mu plus 1mu
\bibliographystyle{IEEEtran}
\bibliography{IEEEabrv,mybib}

\end{document}